\documentclass[sn-mathphys-num]{sn-jnl}
\usepackage{graphicx}%
\usepackage{multirow}%
\usepackage{amsmath,amssymb,amsfonts}%
\usepackage{amsthm}%
\usepackage{mathrsfs}%
\usepackage[title]{appendix}%
\usepackage{xcolor}%
\usepackage{textcomp}%
\usepackage{manyfoot}%
\usepackage{booktabs}%
\usepackage{algorithm}%
\usepackage{algorithmicx}%
\usepackage{algpseudocode}%
\usepackage{listings}%
\numberwithin{equation}{section}
\usepackage{setspace}
\DeclareUnicodeCharacter{0301}{\'{e}}
\allowdisplaybreaks
\usepackage[pagewise]{lineno}

\begin{document}

\title{\textbf{Nonlinear Fractal Histopolation Function}}

\author*[1]{\fnm{Aiswarya} \sur{T}}\email{aiswaryasidhu8113@gmail.com}

\author[1]{\fnm{Srijanani Anurag} \sur{Prasad}}\email{srijanani@iittp.ac.in}

\affil*[1]{\orgdiv{Department of Mathematics and Statistics}, \orgname{Indian Institute of Technology Tirupati}, \orgaddress{\street{Yerpedu P O}, \city{Tirupati}, \postcode{517619}, \state{Andhra Pradesh}, \country{India}}}
\date{}

\newtheorem{proposition}{Proposition}
\newtheorem{theorem}{Theorem}
\newtheorem{definition}{Definition}
\newtheorem{example}{Example}
\newtheorem{remark}{Remark}
\newtheorem{corollary}{Corollary}
\abstract{A novel method for constructing a nonlinear fractal histopolation function associated with a given histogram is introduced in this paper. In contrast to classical fractal interpolation methods, which produce continuous and interpolatory functions, the proposed approach constructs a bounded, Riemann integrable function that is not necessarily continuous but preserves the area of a given histogram. An iterated function system based on Rakotch contractions- a generalisation of Banach contractions- is utilised, thereby extending the theoretical framework for fractal histopolation. Unlike previous formulations, the proposed construction of nonlinear fractal functions allows vertical scaling factors greater than one. The conditions for the nonlinear fractal function to be a solution for the histopolation problem are derived.}

\maketitle

\textbf{Keywords}: Histopolation, Iterated Function Systems, Rakotch Contraction, Fractal Interpolation Function\\
\textbf{MSC Classification}: 28A80, 41A30, 37L30

\section{Introduction} 
Approximation theory is a branch of mathematics exploring methods to represent complex functions or data with more straightforward and manageable models. It aims to find efficient ways to approximate functions or datasets accurately using simpler mathematical constructs. Fractal approximation techniques provide a methodology to precisely approximate functions that exhibit self-similarity and a non-smooth nature.
	 
Interpolation is a commonly used method in approximation theory to find a continuous function whose graph coincides with a given data set. Histopolation is a method closely related to interpolation, but instead of approximating a function directly, histopolation preserves the area of the region below the function's graph. Unlike interpolation, the histopolating function is not necessarily continuous.

Fractal interpolation is a widely used approximation method with diverse practical applications. From image reconstruction to modelling of natural objects, fractal interpolation has proven its versatility. M.F. Barnsley introduced the fractal interpolation function(FIF) as a continuous function that exhibits self-similarity and interpolates a given data set\cite {barn1986}. The construction was based on the iterated function system(IFS) and the Banach contraction principle. Many researchers have worked on different aspects of fractal interpolation functions. Some of these studies constructed fractal interpolation functions of a more general nature, such as Hidden-variable FIFs~\cite{barnsley89_2}, Hermite FIFs~\cite{navas-sebas04}, Coalescence HFIFs~\cite{chand07}, Super FIFs~\cite{srijanani15} and Super CHFIFs~\cite{srijanani21}. There have been studies conducted on properties like smoothness~\cite{gang96}, regularity~\cite{ srijanani13_2}, approximation property~\cite{navascues03,srijanani14_1}, multiresolution analysis~\cite{hardin92,srijanani14_2}, reproducing kernel~\cite{bouboulis11,srijanani19}, node insertion~\cite{kocic03,srijanani13_1}, fractional calculus~\cite{pan14,srijanani17}, and dimension property~\cite{ruan09}. All the above constructions and studies had finite data points. Another extension from finiteness to countable data sets was introduced in~\cite{secelean2003fractal}. Using Rakotch contractions and Rakotch's fixed point theorem, nonlinear FIFs were introduced for finite\cite{ri2018new} and countable data sets\cite{pacurar2021countable}. 
	
All the functions constructed in the works above were continuous and interpolatory. Recently, fractal histopolation functions were constructed by M.F. Barnsley and P. Viswanathan, which differ from FIFs as they are neither continuous (necessarily) nor  interpolatory\cite{barnsley2023histopolating}. The construction was based on the Banach contraction principle.
	
This paper uses Rakotch's fixed point theorem to construct nonlinear fractal histopolation functions. This construction is significant as it generalises the fractal histopolation function. Here, the IFS under consideration will consist of a more significant class of functions, the Rakotch contractions, rather than the small class of Banach contractions. A novel method for incorporating vertical scaling factors greater than one into the construction is presented, thereby enhancing its generality.
	
The structure of the paper is as follows: Some definitions, notations and theorems are discussed in section~\ref{prelims}. A nonlinear fractal function is constructed in section~\ref{Nonlinear fractal functions} as the fixed point of a Read-Bajraktarevi\'{c} operator in the space of bounded functions on a closed and bounded interval. Conditions for the existence of the nonlinear fractal histopolation function as a solution for the given histopolation problem are found in section~\ref{NFHfunctions}.
\section{Preliminary Facts}\label{prelims}
This section contains notations, definitions and theorems necessary for the main result. Consider the metric space $(X,\rho)$.
\begin{definition}[Banach Contraction]
\cite{rudin1964principles} Let $g$ be a function on $X$. $g$ is a Banach contraction if there exists a non-negative number $k<1$ such that
\begin{equation}
    \nonumber \rho(g(x),g(y))\leq k\rho(x,y), \forall \ x,y\in X.
\end{equation}
\end{definition}
\noindent The class of $\psi$-contractions generalises the concept of a Banach contraction on a metric space.
\begin{definition}[$\psi$-Contraction]
\cite{strobin2015attractors} Let $\psi$ be a self map on $\mathbb{R}_{+}$. A function $f$ on $X$ is a $\psi$-contraction if
\begin{equation}
\rho(f(t),f(s))\leq \psi(\rho(s,t)), \quad \forall\ s,t\in X \nonumber.
\end{equation}
\end{definition}
\noindent One of the widely discussed $\psi$-contractions is the Rakotch contraction.
\begin{definition}[Rakotch Contraction]
\cite{ri2018new}Let f be a $\psi$-contraction. Suppose $\psi$ satisfies $\frac{\psi(s)}{s}<1$ and $\frac{\psi(t)}{t}\leq\frac{\psi(s)}{s},\forall \ t>s$. Then, the function $f$ is called a Rakotch contraction.
\end{definition}
\noindent The following examples illustrate that all Banach contractions are Rakotch contractions, but the converse is not true.
\begin{example}
    Let $(X,\rho)$ be a complete metric space and $f:X\rightarrow X$ be a Banach contraction such that $\rho(f(t),f(s))\leq k\rho(t,s), \forall~t,s\in X$ for some $k \in [0,1)$. Let $\psi(t)=kt$ for the same $k$. Then, $\rho(f(t),f(s))\leq \psi(\rho(t,s))$. Hence, $f$ is a Rakotch contraction.
\end{example}
\begin{example}
    Consider $X=[0,1]$ with the Euclidean metric and the function $h(t)=\frac{1}{1+t}, \forall \ t\in X$. $\psi(t)=\frac{t}{1+t}$. Then $\forall\ t,s\in X$,
    \begin{eqnarray*}
        |h(t)-h(s)|=|\frac{1}{1+t}-\frac{1}{1+s}|=|\frac{s-t}{(1+t)(1+s)}|\leq \psi(|t-s|).
    \end{eqnarray*}
    Thus $h$ is a Rakotch contraction. Note that the fixed point of $h$ is $\frac{\sqrt{5}-1}{2}$. Let\\ $k\in[0,1)$ be any fixed constant. Let $s$ be a positive real number such that $s<\min\left\{-1+\frac{1}{k},\frac{\sqrt{5}-1}{2}\right\}$. Then,
    $$|h(0)-h(s)|=|1-\frac{1}{1+s}|=|\frac{s}{1+s}|>ks=k|0-s|.$$
    Thus, $h$ is not a Banach contraction.
\end{example}
\noindent The existence of the fixed point of a Rakotch contraction is established in the following proposition.
\begin{proposition}\label{rfp}
\cite{drakopoulos2021generalised} Consider a Rakotch contraction $f$ on a complete metric space $(X,\rho)$. Then, for any $t\in X$,
\begin{equation}
\lim_{n\rightarrow\infty} f^n(t)=t_0\nonumber
\end{equation} 
and this point $t_0$ is the unique fixed point of f.
\end{proposition}
\begin{definition}[Iterated Function System(IFS)]
    \cite{barnsley2014fractals}An iterated function system, $\mathcal{I}$, is defined as a complete metric space $X$ together with a finite number of contraction maps, $f_1,f_2,\ldots,f_N$ on $X$. This IFS is denoted as\\ $\mathcal{I}=\left\{X;f_j, j=1, 2,\ldots, N\right\}$.
\end{definition}
\begin{definition}[Attractor]
    \cite{barnsley2014fractals} A non-empty compact subset $A$ of $X$ is called an attractor for IFS ~$\mathcal{I}$ if
    \begin{equation*}
        A=\bigcup_{j=1}^N f_j(A).
    \end{equation*}
\end{definition}
\noindent The following proposition provides the existence of the attractor for an iterated function system consisting of Rakotch contractions.
\begin{proposition}
\cite{drakopoulos2021generalised} 
Consider the IFS $\left\{X;f_1, f_2, \dots, f_N\right\}$ on a complete metric space $X$ where each of the $f_j$ is a Rakotch contraction. Then, there exists a unique nonempty compact set $B$ such that
\begin{equation*}
B=\bigcup\limits_{j=1}^{N}f_j(B).
\end{equation*} 
\end{proposition}
\section{Nonlinear Fractal Functions}\label{Nonlinear fractal functions}

Let $N>1$ and $\mathbb{N}_N=\left\{1,2,\ldots, N\right\}$. Consider the interval $I=[t_0,t_N]\subset\mathbb{R}$. Let $\left\{t_0,t_1,\dots ,t_N\right\}\subset I$ with $t_{j-1}<t_j$, for $j\in \mathbb{N}_N$. Let $I_j=[t_{j-1},t_j)$ for $j\in \mathbb 
{N}_{N-1}$ and $I_N=[t_{N-1},t_N]$. Let $l_j:I\rightarrow I$ be the contractive maps 
\begin{equation}
l_j(t)=a_jt+b_j  \label{l_j}
\end{equation}
satisfying $l_j(t_0)=t_{j-1} $ and $l_j(t_N)=t_j$ for all $j\in\mathbb N_N$.

\noindent Set $A=I \times \mathbb{R}$. Define $F_j:A\rightarrow\mathbb{R}$ as
\begin{equation}
F_j(t,x)=c_jt+\delta_js_j(x)+d_j, \label{F_j}
\end{equation}
where $s_j(x)$ are Rakotch contractions with respect to the same function $\psi$ for all $j\in \mathbb N_N$. The coefficients $c_j, d_j$ and $\delta_j$ are constants where $\delta_j$ are chosen as follows.\\If $\sup\limits_{t>0}\frac{\psi(t)}{t}<1$, then there exists a $\beta$ such that $\sup\limits_{t>0}\frac{\psi(t)}{t}<\beta<1$. In this case choose $\delta_j$ such that $\delta=\max\limits_{j\in \mathbb{N}_N}|\delta_j|<\frac{1}{\beta}$. \\If $\sup\limits_{t>0}\frac{\psi(t)}{t}=1$, then choose $\delta_j$ such that $\delta=\max\limits_{j\in \mathbb{N}_N}|\delta_j|<1$.\\The functions $F_j$ are bounded with respect to $t$. 
\begin{remark}
    Unlike the construction of interpolation functions, the join-up conditions are not required. The generality of this construction is ensured by the fact that the vertical scaling factors $\delta_j$ may be any real number satisfying the given constraints. This flexibility distinguishes the approach from earlier work.
\end{remark}
\begin{remark}
The vertical scaling factors can be chosen as variable functions. In such a case $\delta=\max\limits_{j\in\mathbb{N}_N}\left\{\sup\limits_{t\in I}|\delta_j(t)|\right\}$ and $\delta$ must satisfy the same conditions as those required in the case of constant vertical scaling factors.
\end{remark}

Consider the set of real-valued bounded functions $\mathcal{B}(I)$ on $I$ equipped with the norm $||h||_{\infty}=\sup\left\{|h(t)|:t\in I\right\}.$
For $h \in\mathcal{B}(I)$ and $t \in I_j$, the Read-Bajraktarević operator $T$ on $\mathcal{B}(I)$ is defined as 
\begin{eqnarray}
Th(t)&=F_j(l_j^{-1}(t),h(l_j^{-1}(t))) = c_jl_j^{-1}(t)+\delta_js_j(h(l_j^{-1}(t)))+d_j\label{RB}.
\end{eqnarray}
\begin{theorem}\label{Rbfpthm}
The operator $T$ defined on $\mathcal{B}(I)$ by Eq.~\eqref{RB} is a Rakotch contraction and T has a unique fixed point $f\in\mathcal{B}(I)$ satisfying
\begin{equation}\label{fp}
f(t)=c_jl_j^{-1}(t)+\delta_js_j(f(l_j^{-1}(t)))+d_j, \forall \ t\in I_j,~ j\in \mathbb{N}_N.
\end{equation}
\end{theorem}
\begin{proof}
$I$ is a disjoint union of $I_j$'s for $j\in\mathbb N_N$.
Clearly, $Tg\in\mathcal{B}(I)$ since each of the $F_j$'s are bounded with respect to $t$ and $g\in\mathcal{B}(I)$. Hence, the operator $T$ is well-defined. Let $g,g'\in\mathcal{B}(I)$. Consider
\begin{eqnarray*}
d_{\mathcal{B}(I)}(Tg,Tg')&=&\sup\limits_{t\in I}|Tg(t)-Tg'(t)|=\max\limits_{j\in \mathbb{N}_N}\sup\limits_{t\in I_j}|Tg(t)-Tg'(t)|.
\end{eqnarray*}
Applying Eq.~\eqref{RB},
\begin{eqnarray*}
d_{\mathcal{B}(I)}(Tg,Tg')&=&\max\limits_{j\in \mathbb{N}_N}\sup\limits_{t\in I_j}|\delta_j(s_j(g(l_j^{-1}(t)))-s_j(g'(l_j^{-1}(t))))|\\
&\leq&\max\limits_{j\in \mathbb{N}_N}|\delta_j|\max\limits_{j\in \mathbb{N}_N}  \sup_{t\in I_j}\psi(|g(l_j^{-1}(t))-g'(l_j^{-1}(t))|),
\end{eqnarray*}
where $\psi(t):\mathbb{R}_{+}\rightarrow\mathbb{R}_{+}$ is a non decreasing function with $\psi(t) < t\ \forall \  t>0$. As $\psi$ is a non decreasing function, for $j\ \in \mathbb N_N$ and $t\ \in I_{j}$,
\begin{eqnarray*}
\psi(|g(l_{j}^{-1}(t))-g'(l_{j}^{-1}(t))|)&\leq&\psi(\sup_{t\in I_{j}}|g(l_{j}^{-1}(t))-g'(l_{j}^{-1}(t))|)\leq\psi(\sup_{t\in I}|g(t)-g'(t)|)\\
&=&\psi(d_{\mathcal{B}(I)}(g,g')).
\end{eqnarray*}
As $t\in I_{j}$ and $j\in\mathbb N_N$ are arbitrary,
\begin{eqnarray*}
\max\limits_{j\in \mathbb N_N}\sup_{t\in I_j}\psi(|g(l_j^{-1}(t))-g'(l_j^{-1}(t))|)&\leq&\psi(d_{\mathcal{B}(I)}(g,g')).
\end{eqnarray*}
This gives
\begin{eqnarray}
d_{\mathcal{B}(I)}(Tg,Tg')&\leq&\delta~\psi(d_{\mathcal{B}(I)}(g,g')),\label{TRakotch}
\end{eqnarray}
where $\delta=\max\limits_{j\in \mathbb{N}_N}|\delta_j|$. Define a function $\zeta:[0,\infty)\rightarrow[0,\infty)$ as $\zeta(t)=\delta~\psi(t)$.

\noindent From Eq.\eqref{TRakotch}
$$d_{\mathcal{B}(I)}(Tg,Tg')\leq \zeta(d_{\mathcal{B}(I)}(g,g')).$$ With appropriate choices of $\delta_j$'s, $\frac{\zeta(t)}{t}<1, ~\forall t>0$. From the Rakotch condition for $\psi$ it follows that $\frac{\zeta(t)}{t}\leq\frac{\zeta(s)}{s} ~\forall~t>s$. 
Hence, the R.B. operator $T$ is a Rakotch contraction on $\mathcal{B}(I)$. Then, by Proposition~\ref{rfp}, $T$ has a unique fixed point $f$, which satisfies
\begin{equation*}
f(t)=c_jl_j^{-1}(t)+\delta_js_j(f(l_j^{-1}(t)))+d_j.
\end{equation*}
\end{proof}
\begin{theorem}\label{graphoffp}
Let $F_j$ and $l_j$ be as defined in Eq.~\eqref{F_j} and Eq.~\eqref{l_j}. Consider~
\\$\mathcal{I}=\left\{ A; w_j, j\in\mathbb N_N\right\}$ where  $w_j: A\rightarrow A $ are defined by 
\begin{equation}
w_j(t,x)=(l_j(t),F_j(t,x)).             \label{wi}
\end{equation}
Let $C=\max\limits_{j\in\mathbb N_N}|c_j|$, where $c_j$ are constants given in Eq.~\eqref{F_j}. Then, each of the maps $w_j$ is a Rakotch contraction with respect to the metric $d_{\eta}$ given as
\begin{equation*}
d_{\eta}((t,x),(t',x')):=|t-t'|+\eta|x-x'|,
\end{equation*}
where $\eta=\frac{1-\max\limits_{j\in \mathbb N_N}|a_j|}{2(C+1)}$.
Further, a unique attractor exists for $\mathcal{I}$. The attractor is the closure of the graph of the fixed point of the R.B. operator given in Eq.~\eqref{RB}.
\end{theorem}
\begin{proof}
Consider $(t,x),(t',x') \in I\times\mathbb{R}$. Then
\begin{eqnarray*}
    |F_j(t,x)-F_j(t',x')|&=&|c_jt+\delta_js_j(x)+d_j-(c_jt'+\delta_js_j(x')+d_j)|\\
    &\leq&C|t-t'|+\delta|s_j(x)-s_j(x')|\leq C|t-t'|+\zeta(|x-x'|),
\end{eqnarray*}
where $\zeta(t)=\delta~\psi(t)$.
\begin{eqnarray*}
d_\eta(w_j(t,x),w_j(t',x'))&=& d_\eta((l_j(t),F_j(t,x)),(l_j(t'),F_j(t',x')))\\
&=& |l_j(t)-l_j(t')|+\eta|F_j(t,x)-F_j(t',x')|\\
&\leq&(|a_j|+\eta C)|t-t'|+\eta\zeta(|x-x'|)\\
&\leq&\bigg(|a_j|+\eta C+\eta\frac{\zeta(|t-t'|+|x-x'|)}{|t-t'|+|x-x'|}\bigg)|t-t'|\\
& &\hspace{2cm}+\eta\frac{\zeta(|t-t'|+|x-x'|)}{|t-t'|+|x-x'|}(|x-x'|).\\
\end{eqnarray*}
As $\frac{\zeta(s)}{s}<1 $ for all $ 
s>0,$
\begin{eqnarray*}
d_\eta(w_j(t,x),w_j(t',x'))&\leq&(\max\limits_{j\in\mathbb N_N} |a_j|+\eta C+\eta)|t-t'| \\
& &\hspace{2cm}+\eta\frac{\zeta(|t-t'|+|x-x'|)}{|t-t'|+|x-x'|}(|x-x'|)\\
&\leq&\max\{\max\limits_{j\in\mathbb N_N
} |a_j|+\eta C+\eta,\frac{\zeta(|t-t'|+|x-x'|)}{|t-t'|+|x-x'|}\} \times \\
& &\hspace{5cm} d_\eta((t,x),(t',x')).
\end{eqnarray*}
 Since $\eta=\frac{1-\max\limits_{j\in\mathbb N_N}|a_j|}{2(C+1)}$, $\max\limits_{j\in\mathbb N_N}|a_j|+\eta C+\eta<1$. Note that $\frac{\zeta(s)}{s}<1$ for $s>0$. Thus,
\begin{equation*}
\max\{\max\limits_{j\in\mathbb N_N} |a_j|+\eta C+\eta,\frac{\zeta(|t-t'|+|x-x'|)}{|t-t'|+|x-x'|}\}<1.
\end{equation*}
For $s>0$, define
\begin{equation*}
\gamma(s):=\max\{\max\limits_{j\in\mathbb N_N} |a_j|+\eta C+\eta,\frac{\zeta(s)}{s}\}.
\end{equation*}
Since $\gamma(s)$ is a non-increasing function,
\begin{eqnarray*}
d_\eta(w_j(t,x),w_j(t',x'))&\leq& \gamma(d((t,x),(t',x')))d_\eta((t,x),(t',x'))\\
&\leq&\gamma(d_\eta((t,x),(t',x')))d_\eta((t,x),(t',x')).
\end{eqnarray*}
Thus, $w_j$ are Rakotch contractions, and $\mathcal{I}$ has a unique attractor, say $B$.

Let $\mathcal{H}(A)$ be the space comprising nonempty compact subsets of $A$ equipped with the Hausdorff metric.
Define $W:\mathcal{H}(A)\rightarrow\mathcal{H}(A)$ as
\begin{equation*}
W(K)=\bigcup\limits_{j=1}^N w_j(K).
\end{equation*}
Then, $B=W(B)$. 

Let $f$ be the fixed point of the R.B. operator given by Eq.~\eqref{RB} and $G$ be its graph, i.e.,
\begin{equation*}
G=\left\{(t,f(t)):t\in I\right\}.
\end{equation*}
Let $(t,x)\in \bar{G}$. So, there exists a sequence $t_m\in I$ such that $(t_m,f(t_m))\in G$ converges to $(t,x)$. Since $I$ is a disjoint union of $I_j$, there exists a unique $j_m\in\mathbb N_N$ such that $t_m\in I_{j_m}$. So $t_m=l_{j_m}(\tilde{t_m})$ for some $\tilde{t_m}\in I$. Further,
\begin{eqnarray*}
(t_m,f(t_m))=(l_{j_m}(\tilde{t_m}),f(l_{j_m}(\tilde{t_m})))
=(l_{j_m}(\tilde{t_m}),F_{j_m}(\tilde{t_m},f(\tilde{t_m})))=w_{j_m}(\tilde{t_m},f(\tilde{t_m})),
\end{eqnarray*}
 which implies $(t_m,f(t_m))\in W(G)\subset W(\bar{G})$. Thus, $(t,x)\in W(\bar{G})=\overline{W(
G )}$. So, 
\begin{equation}
\bar{G}\subseteq W(\bar{G}).   \label{eq:G1}   
\end{equation}
Conversely, let $(t,x) \in W(\bar{G})=\overline{W(G)}$. Now, there exists a sequence $\{(t_m,x_m)\} \in W(G)$ such that $(t_m,x_m) \rightarrow (t,x)$ as $m \rightarrow \infty$. Since $(t_m,x_m) \in W(G)$, there exists a unique $j_m\in\mathbb N_N$ such that 
\begin{eqnarray*}
(t_m, x_m) = w_{j_m}(\tilde{t_m}, f(\tilde{t_m})) = (l_{j_m}(\tilde{t_m}),F_{j_m}(\tilde{t_m},f(\tilde{t_m}))) = (l_{j_m}(\tilde{t_m}),f(l_{j_m}(\tilde{t_m}))).  
\end{eqnarray*} 
This implies that $(t_m,x_m) = (l_{j_m}(\tilde{t_m}),f(l_{j_m}(\tilde{t_m}))) \in G$ and hence $(t,x) \in \bar{G}$. This gives
\begin{equation}
W(\bar{G})\subseteq \bar{G}. \label{eq:G2}   
\end{equation}
Eq.~\eqref{eq:G1} and Eq.~\eqref{eq:G2} together implies that $W(\bar{G})=\bar{G}$. Since the attractor $B$ is unique which satisfy the equation  $B=W(B)$, $B=\bar{G}$. 
\end{proof}	
\newpage
\begin{example}\label{nlfeg}
    Consider the interval $I=[0,1]$ and the set of points $\left\{0,\frac{1}{2},1\right\}\subset I$. The functions $l_j:I\rightarrow I$ in Eq.~\eqref{l_j} are given by,
    $$l_1(t)=\frac{1}{2}t ~\text{ and }~ l_2(t)=\frac{1}{2}t+\frac{1}{2}.$$ Let $s_1(x)=s_2(x)=\frac{1}{2}sin x$. $s_j(x)$ for $j=1, 2$ are Rakotch contractions with respect to the function $\psi(x)=\frac{1}{2}x$. Let $\delta_1(t)=\frac{3}{2}t$ and $\delta_2(t)=\frac{7}{4}t$.
    Now consider the functions
    \begin{equation*}
        F_1(t,x)=\frac{1}{2}t+\frac{3t}{2}\frac{1}{2}sinx+\frac{1}{3} ~\text{ and }~
        F_2(t,x)=\frac{1}{4}t+\frac{7t}{4}\frac{1}{2}sinx+\frac{1}{6}.
    \end{equation*}
    The fixed point of the R.B. operator is given as
    \begin{equation*}
        f(t)=
        \begin{cases}
            t+\frac{3t}{2}sin(f(2t))+\frac{1}{3} & \text{ if } t\in[0,\frac{1}{2})\\
            \frac{1}{2}t+\frac{7t}{4}sin(f(2t-1))-\frac{23}{24}& \text{ if } t\in[\frac{1}{2},1].
        \end{cases}
    \end{equation*}
    Fig.~\ref{vareg} represents the graph of f:
    \begin{figure}[H]
  \centering
\includegraphics[width=0.9\linewidth]{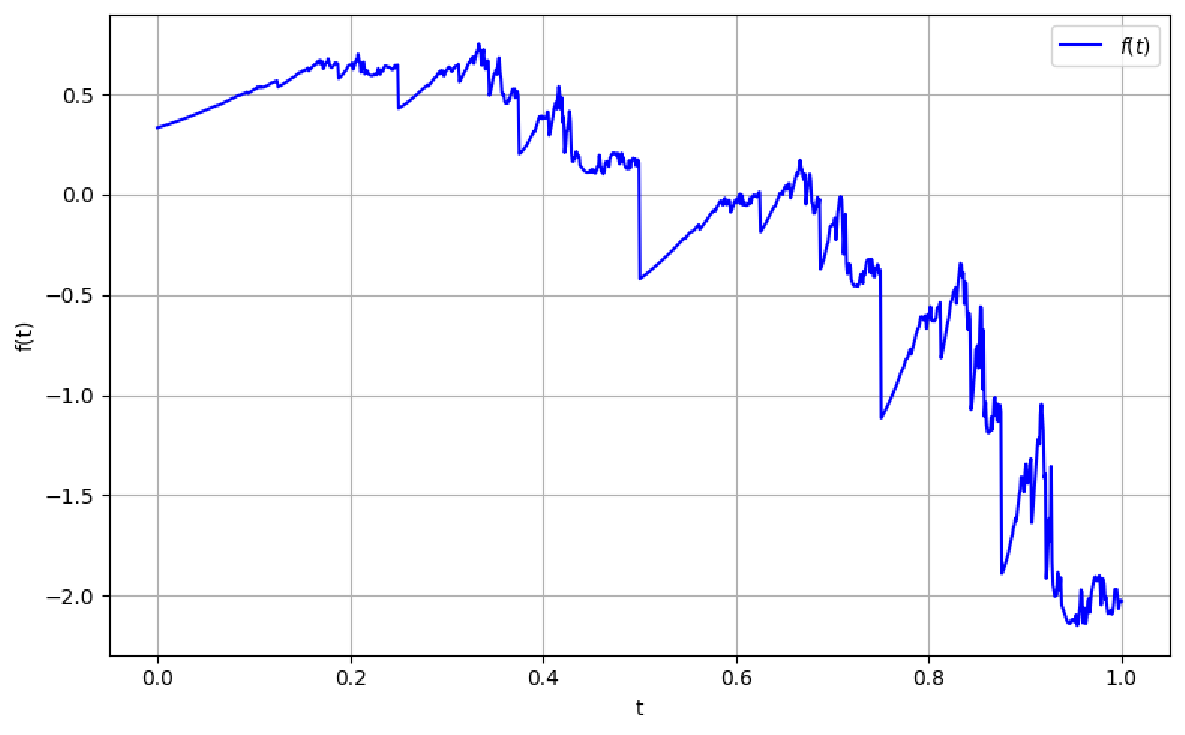}
  \caption{\centering{Graph of $f$} }
  \label{vareg}
\end{figure}%
\end{example}
\noindent Example~\ref{nlfeg} illustrated the generality of the construction. The vertical scaling factors are chosen as variable functions and $\delta>1$.

\begin{definition}\cite{barnsley2023histopolating}
For a real-valued function $f$ from the metric space $(X,\rho)$, the oscillation of the function $f$ in $A\subset X$ is defined as
\begin{equation*}
w_f(A)=\sup_{a,b\in A}|f(a)-f(b)|
\end{equation*}
and oscillation of $f$ at a point $x \in X$ is defined as
\begin{equation*}
\lim_{r\rightarrow0}w_f(B(x,r))
\end{equation*}
where, $B(x,r)$ denotes the open ball with center $x$ and radius $r$.
\end{definition}

\noindent The following theorem establishes the Riemann integrability of the fixed point $f$.

\begin{theorem} \label{riemann}
Let $l_j$ and $F_j$ be as defined in Eq.~\eqref{l_j} and Eq.~\eqref{F_j} respectively. Let $\alpha=\max\limits_{j\in \mathbb N_N}|a_j| < 1$ and $\delta=\max\limits_{j\in \mathbb{N}_N}|\delta_j|<1$ where $a_j$ and $\delta_j$ are the constants as defined in Eq.~\eqref{l_j} and Eq.~\eqref{F_j}.
Then, the fixed point of the R.B. operator, f, is Riemann integrable.
\end{theorem}
\begin{proof}
Theorem~\ref{Rbfpthm} established that the R.B. operator $T$ is a Rakotch contraction and it has a fixed point denoted as 
$f$. Hence $f=\lim\limits_{n\rightarrow\infty}g_n,$ where $g_n=Tg_{n-1}$ for any $g_0\in \mathcal{B}(I)$. Let $g_0$ be a constant function on $I$. Then,
\begin{eqnarray*}
g_1(l_j(t))=Tg_0(t) = c_jt+\delta_js_j(g_0(t))+d_j.
\end{eqnarray*}
The function $g_1$ may have discontinuities at the points of partition $t_1, t_2,\ldots, t_{N-1}$ and consequently $g_n$ have discontinuities at the points $l_{w}(t_k)$ where $w$ is a word of length $n$ and $k\in\left\{1,\ldots, N-1\right\}$. Let $D_n$ be the set of points at which $g_n$ is discontinuous. Then $D_n$ is a finite set. Let $D=\bigcup\limits_{n=1}^{\infty}D_n$. It is a countable set as it is a countable union of finite sets. So $D$ has Lebesgue measure zero.  $D$ consists of points at which $f$ is discontinuous.

 From the construction of the IFS, the maps $F_j$ are Lipschitz in $t$ with Lipschitz constant $c_j$ and Rakotch with respect to $x$. Let $C=\max\limits_{j\in\mathbb N_N}|c_j|$.

\noindent For $I'\subset I$,
\begin{eqnarray} \nonumber
w_f(l_j(I'))&=&\sup_{a,b\in I'}|f(l_j(a))-f(l_j(b))|\\ \nonumber
&=&\sup_{a,b\in I'}|c_j(a-b)+\delta_j(s_j(f(a))-s_j(f(b)))|\\ \nonumber
&\leq&C\sup_{a,b\in I'}|a-b|+\sup\limits_{j\in\mathbb{N}_N}|\delta_j|\sup_{a,b\in I'}|s_j(f(a))-s_j(f(b))|\\ \nonumber
&\leq&C|I'|+\delta\sup_{a,b\in I'}\psi(|f(a)-f(b)|)\\ 
&\leq&C|I'|+\delta\sup_{a,b\in I'}|f(a)-f(b)| \leq C|I'|+\delta w_f(I'),\label{5}
\end{eqnarray}
where $\delta=\max\{|\delta_j|:j\in\mathbb N_N\}<1$. 

Consider the code space $\Omega$ corresponding to $j\in\mathbb N_N$ equipped with the metric $d(\sigma,\sigma')=2^{-p}$ where $p$ is the smallest index for which $\sigma_p\neq\sigma'_p$. This is a compact metric space.  For any $\sigma\in \Omega$, $\sigma|_p=\sigma_1\sigma_2\ldots\sigma_p$ is the finite code of length $p$.\\ From Eq.~\eqref{5},
\begin{eqnarray}\nonumber
w_f(l_{\sigma|_p}(I'))&=&w_f(l_{\sigma_1}\circ l_{\sigma_2}\circ\ldots \circ l_{\sigma_p}(I'))\leq C\alpha^{p-1}|I'|+\delta w_f(l_{\sigma_2}\circ\ldots \circ l_{\sigma_{p-1}}\circ l_{\sigma_p}(I')).\label{osci}
\end{eqnarray}
Repeating the above process recursively,
\begin{equation*}
w_f(l_{\sigma|_p}(I')) \leq 
\begin{cases}
\delta^pw_f(I')+C|I'|\frac{\alpha^p-\delta^p}{\alpha-\delta} & \text{if } \delta<\alpha\\
\delta^pw_f(I')+C|I'|\frac{\delta^p-\alpha^p}{\delta-\alpha} & \text{if } \alpha<\delta\\
\alpha^pw_f(I')+C|I'|p\alpha^{p-1} & \text{if } \alpha=\delta.
\end{cases}
\end{equation*}
Now, $I$ is the attractor of the IFS $\left\{I;l_j,j\in\mathbb N_N\right\}$. For any $t\in I$, $\lim\limits_{p\rightarrow\infty}l_{\sigma_1}\circ l_{\sigma_2}\circ \ldots \circ l_{\sigma_p}(t')=t$ exists and is independent of $t$. The function $\beta:\Omega\rightarrow I$ given by
\begin{equation*}
\beta(\sigma):=\lim_{p\rightarrow\infty}l_{\sigma_1}\circ l_{\sigma_2}\circ \ldots \circ l_{\sigma_p}(t)
\end{equation*}
is continuous and onto. Therefore, for every $t\in I$, $\ \exists\ \sigma\in \Omega$ such that
\begin{equation*}
t=\beta(\sigma)=\bigcap\limits_{p=1}^\infty l_{\sigma|_p}(I).
\end{equation*} 
If $t\notin D$, there is an open interval $V_p$ such that 
\begin{equation*}
t\in V_p \subset l_{\sigma|_p}(I).
\end{equation*}
Let $U_p=l_{\sigma|_p}^{-1}(V_p)$. Since $U_p\subset I$,
\begin{equation}
w_f(V_p)=w_f(l_{\sigma|_p}(U_p))\leq
\begin{cases}\label{oscicases}
\delta^pw_f(I)+C|I|\frac{\alpha^p-\delta^p}{\alpha-\delta} & \text{if } \delta<\alpha\\
\delta^pw_f(I)+C|I|\frac{\delta^p-\alpha^p}{\delta-\alpha} & \text{if } \alpha<\delta\\
\alpha^pw_f(I)+C|I|p\alpha^{p-1} & \text{if } \alpha=\delta.		\end{cases}
\end{equation}
Finally, $ \delta^p\rightarrow 0,\ \alpha^p\rightarrow 0\text{ and } p\alpha^{p-1}\rightarrow 0\text{ as } p\rightarrow\infty $. Choose $p$ large enough such that $w_f(V_p)<\epsilon, \ \forall\ \epsilon>0.$ Thus, $w_f(t)=0$ if $t\notin D$. As $D$ has measure zero and $f$ is continuous on $I\setminus D$, $f$ satisfies Lebesgue's criterion. Thus, $f$ is
Riemann integrable.
\end{proof}	
\begin{remark}The condition $\delta < 1$ is crucial for the validity of Theorem~\ref{riemann}.\\ If $\delta \geq 1$, then the sequence $\delta^p$ does not converge to zero as $p \to \infty$ and the quantity $w_f(V_p)$ cannot be made arbitrarily small. 
\end{remark}
\section{Nonlinear Fractal Histopolation functions}\label{NFHfunctions}
Let $I=[t_0,t_N]\subset\mathbb{R}$ and $\left\{t_0,t_1,\dots, t_N\right\}\subset I$ such that $t_j<t_{j+1}$. Let $F=\left\{y_1,y_2,\dots y_N\right\}$ be the associated histogram, i.e., the frequency of the interval $I_j=[t_{j-1},t_j)$ is $y_j$ for $j=1,2\ldots, N-1$ and $y_N$ is the frequency of $[t_{N-1},t_N]$. The length of the interval $I_j$ is $a_j(t_N-t_0)$. The solution to the problem of histopolation is a function $f$ which satisfies
\begin{equation}
\int\limits_{t_{j-1}}^{t_j} f(t)dt=y_ja_j(t_N-t_0).\label{hdef}
\end{equation}

\noindent For $j\in\mathbb N_N$, let $l_j$ on $I$ be defined as in Eq.~\eqref{l_j} and $F_j: A\rightarrow\mathbb{R}$ be functions which are defined as in Eq.~\eqref{F_j}. Then, by Theorem~\ref{graphoffp} and Theorem~\ref{riemann}, a Riemann integrable function $f\in\mathcal{B}(I)$ exists and it satisfies Eq.~\eqref{fp}.
\begin{theorem}
Let $\left\{t_0,t_1,\ldots, t_N\right\}\subset I=[t_0,t_N]$ with $t_{j-1} < t_j$ for $j\in\mathbb N_N$ and $F=\left\{y_1,y_2,\ldots, y_N\right\}$ be the corresponding histogram. Let $\mathcal{I}=\{A;w_j,j\in\mathbb N_N\}$, where $w_j$ are defined as in Eq.~\eqref{wi} be the IFS. Assume that $c_j$ and $\delta_j$ are fixed according to the assumptions in Theorem~\ref{riemann}. Then, $f$, the fixed point of the R.B. operator in Eq.\eqref{RB}, solves the histopolation problem if and only if $d_j$ satisfies
\begin{equation}
d_j=\frac{2y_j(t_N-t_0)-c_j(t_N^2-t_0^2)-2\delta_j\int\limits_Is_j(f(t))dt}{2(t_N-t_0)}.\label{hcondn}
\end{equation}
\end{theorem}
\begin{proof}
From the histopolation condition in Eq.~\eqref{hdef},
\begin{eqnarray*}
a_j(t_N-t_0)y_j 
&=&\int\limits_{t_{j-1}}^{t_j}(c_jl_j^{-1}(t)+\delta_js_j(f(l_j^{-1}(t)))+d_j)dt\\
&=&c_j\int\limits_{I}ta_jdt+\delta_j\int\limits_{I}s_j(f(t))a_jdt+d_ja_j(t_N-t_0)
\end{eqnarray*}
which on integration becomes
\begin{eqnarray*}
a_j(t_N-t_0)y_j&=&\frac{c_ja_j}{2}(t_N^2-t_0^2)+d_ja_j(t_N-t_0)+\delta_ja_j\int\limits_{I}s_j(f(t))dt.
\end{eqnarray*}
This implies $$d_j=\frac{2y_j(t_N-t_0)-c_j(t_N^2-t_0^2)-2\delta_j\int\limits_Is_j(f(t))dt}{2(t_N-t_0)}.$$

Conversely,  $f$ satisfies 
\begin{equation*}
	f(t)=c_jl_j^{-1}(t)+\delta_js_j(f(l_j^{-1}(t)))+d_j, \forall t\in I_j.
\end{equation*}
So,
\begin{equation}
	\int\limits_{I_j}f(t)dt=\delta_ja_j\int\limits_{I}s_j(f(t))dt+a_jc_j\frac{t_N^2-t_0^2}{2}+d_ja_j(t_N-t_0).\label{conversestep}
\end{equation}
Eq.~\eqref{hcondn} in Eq.~\eqref{conversestep} implies that $\int\limits_{I_j}f(t)dt=a_jy_j(t_N-t_0)$.
\end{proof}
\begin{remark}
Instead of taking the coefficients of $s_j$ as constants $\delta_j$, they can be variable functions of $t$, say $\delta_j(t)$ with the condition that $\max\limits_{t\in I}|\delta_j(t)|<1$. In this case, the if and only if condition for the existence of a solution to the histopolation problem is
\begin{equation*}
    d_j=\frac{2y_j(t_N-t_0)-c_j(t_N^2-t_0^2)-2\int\limits_{I}\delta_j(t)s_j(f(t))dt}{2(t_N-t_0)}.
\end{equation*}
\end{remark}

\begin{example}
Let $\{0,\frac{1}{2},1\}$ be the set with the histogram $\{5,6\}$.
Let $l_j:[0,1]\rightarrow[0,1]$ be
\begin{equation*}
    l_1(t)=\frac{1}{2}t ~\text{ and }~
    l_2(t)=\frac{1}{2}t+\frac{1}{2}.
\end{equation*}
Define $F_j$ as
\begin{equation*}
     F_1(t,x)=\frac{1}{2}t+\frac{1}{2}\frac{1}{(1+x)}+d_1 \text{ and } 
    F_2(t,x)=\frac{1}{4}t+\frac{1}{4}\frac{1}{(1+x)}+d_2.
\end{equation*}
   $\frac{1}{1+t}$ is a Rakotch contraction with respect to the function $\frac{t}{1+t}, t\in\mathbb{R_+}$. Theorem~\ref {Rbfpthm} guarantees that the R.B. operator has a fixed point, $f$, in $\mathcal{B}([0,1])$. If the constants $d_1$ and $d_2$ satisfy Eq.~\eqref{hcondn}, then $f$ is a solution to the histopolation problem. 
The explicit expression of the nonlinear histopolation function, $f$, as a self-referential equation is given by
\begin{equation*}
f(t)= 
\begin{cases}
t+\frac{1}{2(1+f(2t))}+\frac{19}{4}-\frac{1}{2}\int\limits_{0}^1\frac{1}{1+f(t)}dt& \text{if } t\in[0,\frac{1}{2})\\
\frac{1}{2}t+\frac{1}{4(1+f(2t-1))}+\frac{45}{8}-\frac{1}{4}\int\limits_0^{1}\frac{1}{1+f(t)}dt& \text{if } t\in[\frac{1}{2},1].
\end{cases}
\end{equation*}
\end{example}
\section{Conclusion}
\noindent The main outcome of this paper is the construction of a nonlinear fractal histopolation function for a data set using an IFS composed of generalised Rakotch contractions. Earlier constructions had conventionally constrained vertical scaling factors to be less than or equal to 1, but this limitation is overcome in the present work. The criteria for the function that solves the histopolation problem are determined for constants and variable functions as vertical scaling factors.
\section{Declarations}
\subsection*{Funding}
\noindent The first author received financial assistance from the Ministry of Education of India in the form of Research Assistantship.
\subsection*{Competing Interests}
\noindent The authors have no competing interests to declare.
\subsection*{Data Availability}
No data is associated with this work.

\bibliographystyle{unsrt}

\end{document}